\documentclass[11pt]{article}
\usepackage{multirow}
\usepackage{amsmath, amssymb, amsthm}
\usepackage{graphicx}
\usepackage{subfig}
\usepackage{color}

\makeatletter
\newtheorem*{rep@theorem}{\rep@title}
\newcommand{\newreptheorem}[2]{%
\newenvironment{rep#1}[1]{%
 \def\rep@title{#2 \ref{##1}}%
 \begin{rep@theorem}}%
 {\end{rep@theorem}}}
\makeatother

\newtheorem{defn}{Definition}
\newtheorem{thm}{Theorem}
\newtheorem*{theorem*}{Theorem}
\newtheorem{prop}{Proposition}
\newreptheorem{prop}{Proposition}
\newtheorem{cor}{Corollary}
\newreptheorem{cor}{Corollary}

\newtheorem{remark}{Remark}

\def\inprob{\xrightarrow{p}}
\def\law{\xrightarrow{d}}

\begin{document}

\long\def\symbolfootnote[#1]#2{\begingroup%
\def\thefootnote{\fnsymbol{footnote}}\footnote[#1]{#2}\endgroup}
\newcommand{\keywordsname}{{\small \bf{Keywords}: }}

\title{The Asymptotic Distribution\\ of the \\Determinant of a Random Correlation Matrix}

\author{\begin{tabular}[t]{cc}
A.M. Hanea$^{a,}$ \footnote{Correspondence to: A.M. Hanea, CEBRA, University of Melbourne, Parkville, VIC 3010, Australia, email: anca.hanea@unimelb.edu.au} \   \& G.F. Nane $^{b}$
\end{tabular}\\
\small $^{a}$ \emph{Centre of Excellence for Biosecurity Risk Analysis, University of
Melbourne, Australia}\\
\small $^{b}$ \emph{Delft Institute of Applied Mathematics, Technical University of Delft}
}

\date{}
\maketitle

\begin{abstract}

Random correlation matrices are studied for both theoretical interestingness and importance for applications. The author of \cite{Holmes91} is interested in their interpretation as covariance matrices of purely random signals, the authors of \cite{Qiu2004} employ them in the generation of random clusters for studying clustering methods, whereas the authors of \cite{Johnson1980} use them for studying subset selection in multiple regression, etc. The determinant of a matrix is one of the most basic and important matrix functions, and this makes studying the distribution of the determinant of a random correlation matrix of paramount importance. Our main result gives the asymptotic distribution of the determinant of a random correlation matrix sampled from a uniform distribution over the space of $d \times d$ correlation matrices. Several spin-off results are proven along the way, and an interesting connection with the law of the determinant of general random matrices, proven in \cite{Nguyen2012}, is investigated.

\noindent{\footnotesize \keywordsname  random correlation matrix, determinant, beta distribution, random matrix, regular vines}
\end{abstract}

\renewcommand{\thefootnote}{\arabic{footnote}}

\section{Introduction}\label{intro}

The authors of \cite{Joe2006} and \cite{Lewandowski2009} have studied extensively the problem of generating random correlation matrices uniformly from the space of positive definite correlation matrices. In \cite{Joe2006} it is shown that since a $d-$dimensional positive definite correlation matrix $R=(\rho_{ij})_{i,j=1,\ldots,d}$ can be parametrised in terms of correlations $\rho_{i,i+1}$ and partial correlations $\rho_{i,j;i+1,...,j-1}$ for $(j-i)\geq 2$, and these parameters can independently take values in the interval $(-1,1)$, one can generate a random correlation matrix by choosing independent distributions $F_{ij}$, $1 \leq i< j\leq d$ for these parameters.

Appropriate choices for $F_{ij}$ lead to a joint density for $\rho_{ij : 1 \leq i < j \leq d}$ that is proportional to $det(R)^{\eta - 1}$, where $\eta > 0$. In \cite{Lewandowski2009} it is shown that in this case the joint density is invariant to the order of indexing of variables for the partial correlations, and each $\rho_{ij}$ marginally has a $Beta\left(\eta - 1 + \frac{d}{2}, \eta - 1 + \frac{d}{2}\right)$ distribution on $(-1, 1)$. The uniform density over the set of positive definite correlation matrices is obtained when $\eta = 1$. The proof of this result is formulated using the notion of the partial correlation regular vines. Vines were introduced in \cite{Cooke_97} and
\cite{Bedford_2002}. A vine on $d$ variables is a nested set of
trees. The edges of the $j^{\underline{th}}$ tree are the nodes of
the $(j+1)^{\underline{th}}$ tree. A \emph{regular} vine on $d$ variables is a vine in
which two edges in tree $j$ are joined by an edge in tree $j+1$ only if these edges share a common node. More formally:

\begin{defn}
$\mathcal{V}$ is called a regular vine on $d$ elements if:
\begin{enumerate}
\item $\mathcal{V}=(T_{1},\dots,T_{d-1})$; \item $T_{1}$ is a tree
with nodes $N_{1}=\{1,\dots,d\}$, and edges $E_{1}$ and for
$i=2,\dots,d-1$, $T_{i}$ is a tree with nodes $N_{i}=E_{i-1}$;
\item For $i=2,\dots,d-1$, ${a,b} \in E_{i}$, $\#a \bigtriangleup
b=2$, where $\bigtriangleup$ denotes the symmetric difference. In other words if $a$ and $b$ are nodes of
$T_{i}$ connected by an edge in $T_{i}$, where $a=\{a_{1}, a_{2}\}$, $b=\{b_{1}, b_{2}\}$, then exactly one of the $a_{i}$ equals one of the $b_{i}$
\end{enumerate}
\end{defn}

\noindent For each edge of the vine we distinguish a \emph{constraint}, a \emph{conditioning}, and a \emph{conditioned} set. Variables reachable from an edge, via the membership relation, form its constraint set. If two edges are joined by an edge in the next tree the intersection and symmetric difference of their constraint sets give the conditioning and conditioned sets, respectively.\\
\indent Each regular vine edge may be associated with a partial correlation. A \emph{complete partial correlation vine specification}
is a regular vine with a partial correlation specified for each
edge. A partial correlation vine specification does not uniquely specify a joint distribution\footnote{Moreover a given set of marginal distributions may not be consistent with a given set of partial correlations.}, but there is a joint distribution satisfying the specified information \cite{Bedford_2002}.\\
The property of vines that plays a crucial role in the starting point of this paper is given in the next theorem \cite{Kurowicka2006LA}.

\begin{thm}
Let $D_d$ be the determinant of a $d \times d$ correlation matrix, $R$, of variables $X_1,\cdots, X_d$, with $D_d>0$. For any partial correlation vine:
\begin{equation}\label{factdetvine}
D_d=\prod_{e \in E(\mathcal{V})}{\left(1-
\rho_{e_{1},e_{2};\mathcal{C}_{e}}^{2}\right)},
\end{equation}
\noindent where $E(\mathcal{V})$ is the set of edges of the vine $\mathcal{V}$, $\mathcal{C}_{e}$ denotes the conditioning set associated with edge $e$, and $\{e_1,e_2\}$ is the conditioned set of $e$.
\end{thm}

\noindent Vines are actually a way of factorising the determinant of the correlation matrix in terms of partial correlations. As mentioned earlier, the uniform density over the set of positive definite correlation matrices is invariant to the order of indexing of variables for the partial correlations, and each $\rho_{ij}$ marginally follows a $Beta\left(\frac{d}{2}, \frac{d}{2}\right)$ distribution on $(-1, 1)$.\\

\begin{remark} It is worth mentioning that if the $\rho_{ij}$'s would be independent, then they would individually approach $0$, as the dimension approaches infinity. As a consequence the determinant of the correlation matrix would then approach $1$. The main theorem from Section \ref{teoreme} reveals a completely different behaviour of the determinant of a random correlation matrix.
\end{remark}

\indent It is shown that each partial correlation $\rho_{i,j;K}$ from factorisation (\ref{factdetvine}) has a $Beta\left(\frac{d-k}{2}, \frac{d-k}{2}\right)$ distribution on $(-1, 1)$, where $k$ is the cardinality of the conditioning set $K$. It follows that each $(\rho_{i,j;K})^2$ has a $Beta\left(\frac{1}{2}, \frac{d-k}{2}\right)$ distribution on $(0, 1)$ and each $\left(1-\rho_{i,j;K}^2\right)$ is distributed according to a $Beta\left(\frac{d-k}{2}, \frac{1}{2}\right)$ distribution. Rearranging the terms from factorisation (\ref{factdetvine}), we rewrite $D_d$ in terms of independent $Beta\left(\frac{d-k}{2}, \frac{1}{2}\right)$ variables on $(0, 1)$ as:

\begin{equation}\label{det_as2prod}
D_d = \prod_{j=0}^{d-2}\prod_{i=j}^{d-2}\ B_i,
\end{equation}

\noindent where $B_i \sim Beta\left(1+\frac{i}{2}, \frac{1}{2}\right)$. \\

Throughout the paper, we will denote by $B_i$ Beta distributed random variables on $(0,1)$. Although determined by $i$(or $j$), the parameter of the Beta variables can differ. This abuse of notation will be compensated by an increased clarity of the exposition.

\section{Main Results}\label{teoreme}

\begin{thm}\label{d_1prod}
For a uniform distribution over the space of $d \times d$ correlation matrices, the marginal
distribution of each correlation is $Beta\left(\frac{d}{2}, \frac{d}{2}\right)$ on $(-1, 1)$. Take $R$ to be such a correlation matrix and let $D_d$ be its determinant. Then $D_d$ can be written as a product of $(d-1)$ independent Beta distributed random variables:

\begin{equation}\label{det_as1prod}
D_d = \prod_{j=1}^{d-1}\ B_j,
\end{equation}
\noindent where $B_j \sim Beta\left(\frac{j+1}{2}, \frac{d-j}{2}\right), j=1, \ldots, d-1.$
\end{thm}

\begin{proof}
\noindent Let $D_d$ be represented as in equation \eqref{det_as2prod}. The variables  $B_i \sim Beta\left(1+\frac{i}{2}, \frac{1}{2}\right)$ are independent. Let $a_i=1+\frac{i}{2}$ and $b_i=\frac{1}{2}$. It follows that $a_{i+1}=a_i+b_i$. The authors of \cite{NadarajahGupta} prove that for such variables:
\begin{equation}\nonumber
\prod_{i} B_i \sim Beta \left(a_1, \sum_{i} b_i\right).
\end{equation}

\noindent Using the above property we reduce the second product in \eqref{det_as2prod} to one variable with known distribution:

\begin{equation}\nonumber
\prod_{i=j}^{d-2}\ B_i \sim Beta\left(1+\frac{j}{2}, \sum_{i=j}^{d-2} \frac{1}{2}\right) = Beta\left(1+\frac{j}{2}, \frac{d-j-1}{2}\right).
\end{equation}

\noindent Then:

\begin{equation}\nonumber
D_d = \prod_{j=1}^{d-1} B_j,
\end{equation}

\noindent where $ B_j \sim Beta\left(\frac{j+1}{2}, \frac{d-j}{2}\right)$.
\end{proof}

\begin{remark}
\noindent Observe that the sum of the Beta distribution parameters in Theorem \ref{d_1prod} is constant and equal to $S = \frac{d+1}{2}$; then the expression of the determinant in equation \eqref{det_as1prod} can be written as the product of  $B_j \sim Beta\left(S-\frac{j}{2}, \frac{j}{2}\right)$ distributed variables, $j=1,\ldots,d-1$.
\end{remark}

\vspace{0.2in}

\noindent Rearranging the terms and looking separately at odd and even values of $d$ we obtain:

\begin{equation*}
D_d   = \begin{cases} B_d \cdot \prod_{j=1}^{\lfloor \frac{d-2}{2} \rfloor}\left( B_j \cdot \widetilde{B_j}\right)   & , d=2k\\
             B_d \cdot B_k \cdot \prod_{j=1}^{\lfloor \frac{d-2}{2} \rfloor}\left( B_j \cdot \widetilde{B_j} \right)   & , d=2k+1,
           \end{cases}
\end{equation*}

\noindent where $B_d \sim Beta\left(\frac{d}{2},\frac{1}{2}\right), B_k \sim Beta\left(\frac{k+1}{2},\frac{k+1}{2}\right), B_j \sim Beta\left(\frac{j+1}{2}, \frac{d-j}{2}\right)$ and $\widetilde{B_j} \sim Beta\left(\frac{d-j}{2}, \frac{j+1}{2}\right)$.

\vspace{0.2in}

\begin{cor}\label{cor1}
Consider the $d \times d$ correlation matrices such that each $\rho_{ij}$ marginally has a Beta distribution on $(-1, 1)$, with both parameters equal to $\left(\eta - 1 + \frac{d}{2}\right)$, for $\eta > 0$. Then the determinant $D_d$ of
such correlation matrix can be written as a product of $(d-1)$ independent Beta distributed random variables $B_j$, i.e. $D_d = \prod_{j=1}^{d-1}\ B_j$, where $B_j \sim Beta\left(\eta + \frac{j-1}{2}, \frac{d-j}{2}\right)$.
\end{cor}

\vspace{0.2in}

\begin{thm}[Main Theorem]
\label{mainth}
For $d>0$, consider the uniform distribution over the space of $d \times d$ correlation matrices. Let $R_d$ be a random correlation matrix and let $D_d$ be its determinant. Then, as $d\rightarrow\infty$,
\[
D_d^{1/d} \inprob \frac{1}{e}
\]
\end{thm}

\begin{proof}
Consider the expression for $D_d$ given in equation \eqref{det_as1prod}, where  $B_j \sim Beta\left(\frac{j+1}{2}, \frac{d-j}{2}\right)$. The variables $B_j$ are independent, so the expectation of  $D_d^{1/d}$ can be calculated as a product of expectations of $Beta$ variables raised to the power $\frac{1}{d}$.\\
 \indent Let $\alpha = \frac{j+1}{2}$, $\beta = \frac{d-j}{2}$, $S = \alpha+\beta=\frac{d+1}{2}$. Let $B (\alpha,\beta)$ denote the Beta function, and $\Gamma(\cdot)$ denote the Gamma function. Then:

\begin{align*}
E\left(D_d^{1/d}\right) &= E \left(\prod_{j=1}^{d-1} B_j^{1/d}\right)\\
&=\prod_{j=1}^{d-1} E\left(B_j^{1/d}\right)\\
&= \prod_{j=1}^{d-1}\frac{B(\alpha+\frac{1}{d}, \beta)}{B(\alpha,\beta)} = \prod_{j=1}^{d-1} \frac{\Gamma(\alpha + \frac{1}{d}) \cdot \Gamma(\alpha + \beta)}{\Gamma(\alpha ) \cdot \Gamma(\alpha + \beta + \frac{1}{d})}\\
&= \left[ \frac{\Gamma(S)}{\Gamma(S+\frac{1}{d})} \right]^{d-1} \cdot  \prod_{j=1}^{d-1} \frac{\Gamma\left(\frac{j+1}{2}+\frac{1}{d}\right)}{\Gamma(\frac{j+1}{2})}.
\end{align*}

\noindent Using results originally formulated in \cite{erdelyi1951} and refined in \cite{Laforgia2012}, we can write:

\begin{equation}
\label{frac_gamma}
\frac{\Gamma(z+a)}{\Gamma(z+b)} = z^{a-b} \left[1+ O\left(z^{-1}\right)\right],
\end{equation}

\noindent for $a,b \geq 0$ and $z\to\infty$.\\

\begin{align*}
E\left(D_d^{1/d}\right) &= \left(S^{-\frac{1}{d}}\right)^{d-1} \cdot \prod_{j=1}^{d-1} \left( \frac{j+1}{2}\right)^{\frac{1}{d}}\cdot\left[1+ O\left(z^{-1}\right)\right)]{d-1}\\
&=\frac{2^{\frac{d-1}{d}}}{(d+1)^{\frac{d-1}{d}}} \cdot \frac{(d!)^{\frac{1}{d}}}{2^{\frac{d-1}{d}}}\cdot\left[+ O\left(z^{-1}\right)\right]{d-1}.
\end{align*}

\noindent Furthermore, by using Stirling's approximation,

\begin{equation}
\label{stirling}
d!= \sqrt{2 \pi d}\left(\frac{d}{e}\right)^d \left[ 1+O\left(d^{-1}\right) \right],
\end{equation}
\noindent then

\begin{align}
\label{exp_det_1/d}
E \left(D_d^{1/d}\right) &= \frac{d^{\frac{1}{2d}+1}}{(d+1)^{1-\frac{1}{d}}}\cdot \left(\sqrt{2\pi}\right)^{\frac{1}{d}} \cdot \frac{1}{e}\cdot\left[ 1+O\left(d^{-1}\right) \right]^{\frac{d-1}{d}}\\
&=\frac{d}{d+1} \cdot d^{\frac{1}{2d}}\cdot (d+1)^{\frac{1}{d}}  \cdot \left(\sqrt{2\pi}\right)^{\frac{1}{d}} \cdot \frac{1}{e}\cdot\left[ 1+O\left(d^{-1}\right) \right]^{\frac{d-1}{d}}.\nonumber
\end{align}

\noindent It follows immediately that
\begin{equation}
\label{exp_det_lim}
\lim_{d\to\infty}E\left(D_d^{1/d}\right)= \frac{1}{e}.
\end{equation}
\noindent Similarly,
\begin{equation}
\label{exp_det_sq}
\lim_{d\to\infty} E\left[\left(D_d^{1/d}\right)^2\right]=\lim_{d\to\infty} E\left(D_d^{2/d}\right)= \frac{1}{e^2}.
\end{equation}
By \eqref{exp_det_lim} and \eqref{exp_det_sq},
\begin{equation}
\label{var_det}
\lim_{d\to\infty} var\left(D_d^{1/d}\right)=\lim_{d\to\infty} \left[E\left(D_d^{1/d}\right)\right]^2 - \lim_{d\to\infty} E\left[\left(D_d^{1/d}\right)^2\right]=0.
\end{equation}
By \eqref{exp_det_lim} and \eqref{var_det}, and by using Chebyshev's inequality, the proof is complete.

\end{proof}

\noindent Since convergence in probability implies convergence in distribution, the following result emerges.

\begin{cor}\label{conv_distr}
For a uniform distribution over the space of $d \times d$ correlation matrices, consider a random correlation matrix $R_d$ and let $D_d$ be its determinant. Consider the sequence of random variables $(D_d)^{1/d}$, for $d=1,2,\ldots$. Then, this sequence converges in distribution to a degenerate random variable $D=\frac{1}{e}$,
\[
D_d^{1/d}\law \frac{1}{e}.
\]
\end{cor}

\noindent The result of Theorem \ref{mainth} also holds for the space of correlation matrices where each element marginally distributed according to a $Beta\left(\eta - 1 + \frac{d}{2}, \eta - 1 + \frac{d}{2}\right)$ distribution on $(-1, 1)$. This is stated in the corollary below and the proof can be found in the Appendix.

\begin{cor}\label{main_cor}
Consider the $d \times d$ correlation matrices such that each $\rho_{ij}$ marginally follows a Beta distribution on $(-1, 1)$, with both parameters equal to $\left(\eta - 1 + \frac{d}{2}\right)$, for $\eta > 0$. Let $R_d$ be such a matrix and $D_{d}$ its determinant. Then $D_d^{1/d} \to \frac{1}{e}$, as $d\rightarrow\infty$, in probability.
\end{cor}

\begin{remark}
Each correlation in the correlation matrix marginally has a $Beta\left(\frac{d}{2}, \frac{d}{2}\right)$ distribution on $(-1, 1)$. When $d\to\infty$, each $Beta\left(\frac{d}{2}, \frac{d}{2}\right)$ distribution will approach $0$. In other words, if the entries in the correlation matrix would be independent, the determinant would approach $1$ as the dimension increases. The above theorem reveals an opposite behaviour, namely the convergence of the determinant towards its other bound, namely $0$. This behaviour is induced by the dependence amongst the entries of the correlation matrix.
\end{remark}

\section{Spin-off Results}

The asymptotic behaviour of the $d^{th}$ root of the first two moments of $D_d$ is formulated in the following proposition. The proof of this proposition can be found in the Appendix.

\begin{prop}\label{prop1}
For a uniform distribution over the space of $d \times d$ correlation matrices, the marginal
distribution of each correlation follows a $Beta(d/2, d/2)$ distribution on $(-1, 1)$, and the determinant $D_d$ of
a correlation matrix follows the expression from equation (\ref{det_as1prod}). Then:
\begin{enumerate}
\item $\displaystyle{\lim_{d\longrightarrow\infty}} \left(E\left(D_d\right)\right)^{\frac{1}{d}}= \frac{1}{e} = \lim_{d\longrightarrow\infty} E\left(D_d^{1/d}\right)$
\item $\displaystyle{\lim_{d\longrightarrow\infty}} \left(var\left(D_d\right)\right)^{\frac{1}{d}} = \frac{1}{e^2} = \lim_{d\longrightarrow\infty}\left[E\left(D_d\right)^{\frac{1}{d}}\right]^2$.
\end{enumerate}
\end{prop}

\noindent The first equality from the above proposition suggests a linear behaviour caused by the degenerate distribution of the scaled determinant. The same behaviour will be later observed in a different context.

\begin{cor}\label{mom_det}
The first two moments of the distribution of the determinant of a $d \times d$ random correlation matrix can be written as follows, for very large $d$:

\begin{enumerate}
\item $E(D_d) = \left[E\left(D_d^{1/d}\right)\right]^d + O\left(\frac{1}{d}\right)$
\item $var(D_d) = \left[E(D_d)\right]^2 + O\left(\frac{1}{d}\right)$.
\end{enumerate}
\end{cor}

\noindent The first moment of $D_d$ can be calculated exactly as a function of the dimension $d$, whereas it is easier to calculate only an approximation of the second moment. Their respective expressions are given in the following proposition (see Appendix for proofs).

\begin{prop}\label{prop2}
For a uniform distribution over the space of $d \times d$ correlation matrices, the marginal
distribution of each correlation is $Beta\left(\frac{d}{2}, \frac{d}{2}\right)$ on $(-1, 1)$ and the determinant $D_d$ of
a correlation matrix follows the expression from equation (\ref{det_as1prod}). Then:
\begin{enumerate}
\item $E(D_d)=\frac{d!}{(d+1)^{d-1}}$
\item $var(D_d) = \frac{1}{6} \cdot \frac{(d+1)!}{(d+1)^d} \cdot \frac{(d+3)!}{(d+3)^d}-\left[\frac{(d+1)!}{(d+1)^d}\right]^2$.
\end{enumerate}
\end{prop}

\noindent Form a completely different perspective, it is interesting to notice that $E(D_d)$ can be expressed as an elementary norm \cite{Hardy_L_P}.

\begin{remark}\label{PHL}
Let $E(D_d)$ be approximated as in Corollary \ref{mom_det} (see Appendix). This approximation is the continuous version of the elementary r-norm, where $r=\frac{1}{d}$. Using Theorem 187 from \cite{Hardy_L_P}, one can prove:\\
\begin{equation*}
E(D_d) = e^{E \left[ \ln(D_d)\right]} + O\left(\frac{1}{d}\right).
\end{equation*}
\end{remark}

\noindent More results about the logarithm of the determinant of the correlation matrix are investigated in the following section.

\section{The Logarithm of the Determinant of a Random Correlation Matrix}

Since sums of independent random variables play a more important role than products, and they have been studied more thoroughly, it would be convenient to investigate the behavior of the logarithm of the determinant of the correlation matrix, rather than of the determinant itself. Let us denote $Y_d = \ln{D_d}$ and look at the first two moments of this new random variable.

Consider $B_j \sim Beta\left( \frac{j+1}{2}, \frac{d-j}{2}\right)$, $d=2k$, $\psi$ the digamma function, and $\gamma$ the Euler - Mascheroni constant\footnote{The calculations will follow the same lines for an odd dimension.}.

\begin{align*}
E(Y_d)&= \sum_{j=1}^{d-1} E\left(\ln B_j \right)=  \sum_{j=1}^{d-1} \left[ \psi\left( \frac{j+1}{2}\right) - \psi\left( \frac{d+1}{2}\right)\right]\\
&= \psi(1)+\dots +\psi(k)+\psi\left(1+\frac{1}{2} \right)+ \dots +\psi\left((k-1)+\frac{1}{2} \right) - (d-1)\psi\left( \frac{d+1}{2}\right)\\
&= -\gamma(d-1)-\ln(2)(d-2)+\sum_{i=1}^{k-1}\left[ \frac{4i-1}{(2i-1)i}(k-i)\right]-(d-1)\psi \left(k + \frac{1}{2} \right).
\end{align*}

\noindent Rewriting  $\psi\left(k +\frac{1}{2}\right)$ as $\gamma -2\ln(2)+ \sum_{i=1}^k \frac{2}{2i-1}$ and rearranging and cancelling terms we obtain:

\begin{equation}
\label{Y_d}
E(Y_d)= -d\left[\sum_{i=1}^k \left(\frac{1}{2i-1} - \frac{1}{2i} \right) - \ln(2) + 1\right] + \sum_{i=1}^k \frac{1}{2i-1}.
\end{equation}
We will further provide lower and upper bounds for $E(Y_d)$. Firstly, note the the first sum is a partial sum of an alternating series that converges to $\ln(2)$:

\begin{equation*}
\sum_{i=1}^k \left(\frac{1}{2i-1} - \frac{1}{2i} \right)=\sum_{i=1}^{2k}\frac{(-1)^{k-1}}{k}\equiv S_{2k}.
\end{equation*}
The partial sum $S_{2k}$ approximates $\ln(2)$ with an error that can be bounded be the next term in the series, that is:

\begin{equation}
\label{bound_sk}
\left| S_{2k}-\ln(2) \right| \leq \frac{1}{2k+1}.
\end{equation}
For the second sum in \eqref{Y_d}, we use that, for any $N>1$:

\begin{equation*}
\ln(N+1)<\sum_{i=1}^N\frac{1}{i}\leq 1+\ln(N).
\end{equation*}
Then:

\begin{equation*}
\sum_{i=1}^k \frac{1}{2i-1}>1+\sum_{i=1}^k \frac{1}{2i}>1+\frac{1}{2}\ln(k+1).
\end{equation*}
Similarly,

\begin{equation*}
\sum_{i=1}^k \frac{1}{2i-1}=\sum_{i=0}^{k-1} \frac{1}{2i+1}<1+\sum_{i=1}^{k-1}\frac{1}{2i}\leq 1+\frac{1}{2}[1+\ln(k-1)].
\end{equation*}
This gives that:

\begin{equation}
\label{bound_oddseries}
1+\frac{1}{2}\ln(k+1)<\sum_{i=1}^k \frac{1}{2i-1}\leq 1+\frac{1}{2}[1+\ln(k-1)].
\end{equation}
By \eqref{bound_sk}:

\begin{equation*}
-\frac{1}{2k+1}+1\leq S_{2k}-\ln(2)+1\leq\frac{1}{2k+1}+1,
\end{equation*}
which gives:

\begin{equation*}
-d\left( \frac{1}{d+1}+1 \right)\leq -d[S_{2k}-\ln(2)+1]\leq -d\left(-\frac{1}{d+1}+1\right).
\end{equation*}
Together with \eqref{bound_oddseries}, we get:

\begin{equation}
\label{lu_bound_E_Y_d}
-d\left( \frac{1}{d+1}+1 \right)+ 1+\frac{1}{2}\ln\left(\frac{d}{2}+1\right)< E(Y_d)\leq -d\left(-\frac{1}{d+1}+1\right)+1+\frac{1}{2}\left[1+\ln\left(\frac{d}{2}-1\right)\right].
\end{equation}
To get an approximation for $E(Y_d)$, we can use the result in Theorem 3.2 in \cite{Apostol_1976}:

\begin{equation}
\label{harmonic_no}
\sum_{i=1}^n\frac{1}{i}=\ln(n)+\gamma+O\left(\frac{1}{n}\right).
\end{equation}
This gives:

\begin{equation}
E(Y_d)=-d+\frac{1}{2}\ln\left(\frac{d}{2}\right)+1+\frac{\gamma}{2}+o(1).
\end{equation}
It would seem that we could use the approximation $E(Y_d)=-d+o(1)$. Nonetheless, the simulations in the following section will show that this approximation is not accurate.
\begin{remark}
Using Remark \ref{PHL} and the above result we obtain: $E(D_d) = e^{E \left( Y_d\right)} + O\left(\frac{1}{d}\right) \to 0$, as $d \to \infty$.
\end{remark}

\noindent To calculate the variance of $Y_d$ we need the first derivative of the digamma function, also known as the trigamma function, denoted by $\psi_1$. Then:

\begin{align*}
var(Y_d)&= \sum_{j=1}^{d-1} var\left(\ln B_j \right)=  \sum_{j=1}^{d-1} \left[ \psi_1\left( \frac{j+1}{2}\right) - \psi_1\left( \frac{d+1}{2}\right)\right]\\
&=k\frac{\pi^2}{6}+(k-1)\frac{\pi^2}{2}-\sum_{i=1}^{k}\left[(k-i)\left(\frac{1}{i^2}+\frac{4}{(2i-1)^2}\right)\right] - (d-1)\psi_1\left(\frac{d+1}{2} \right).
\end{align*}
\noindent By rewriting $\psi_1\left(\frac{d+1}{2} \right)$  as $\psi_1\left(\frac{1}{2}\right)-4\sum_{i=1}^k \frac{1}{(2i-1)^2}$, substituting for the known values of $\psi_1$, reducing and manipulating remaining terms, we obtain:

\begin{equation*}
var(Y_d)= \ln{k} + \gamma + \frac{1}{2k} + \frac{\pi^2}{4} + 2\sum_{i=1}^k \frac{1}{2i-1}.
\end{equation*}
By using \eqref{bound_oddseries}, we obtain the lower and upper bounds for $var(Y_d)$:

\begin{equation}
\label{lu_bound_var_Y_d}
\ln\left(\frac{d}{2}\right) + \gamma + \frac{1}{d}+ \frac{\pi^2}{4}+2+\ln\left(\frac{d}{2}+1\right)<var(Y_d)\leq
\ln\left(\frac{d}{2}\right) + \gamma + \frac{1}{d}+ \frac{\pi^2}{4}+3+\ln\left(\frac{d}{2}-1\right).
\end{equation}
It can be easily shown that $var(Y_d) = 2\ln\left(\frac{d}{2}\right)+o(1)$.\\

Even though the expression of the first two moments of $Y_d$ do not look particularly promising, they will be employed in proving the central limit theorem (CLT) for independent, but not necessarily identically distributed, random variables. However, before diving into the proof, we present a simulation based confirmation of such a result, which provides an interesting link with results for the determinant of random correlation matrices in general.

\subsection{The law of the determinant of random matrices}

The authors of \cite{Nguyen2012} proved that for a $d-$dimensional random matrix $A_d$ whose entries $a_{ij}$ are independent real random variables with mean zero and variance one, the logarithm of $|det (A_d)|$ satisfies the central limit theorem. More precisely:
$$\sup_{x\in R} |P\left(\frac{\log (|det A_d|)- 1/2 \log(d-1)!}{\sqrt{1/2 \log d}}\le x \right) -\Phi(x)| \le \log^{-1/3 +o(1)} d.$$\\
Since the entries of a random correlation matrix are definitely not independent, we cannot use this result as it is. So far, we only have an indication of the order of the first two moments of $Y_d$.

The entries in our correlation matrix $R=(\rho_{ij})_{i,j=1,\ldots,d}$ are $Beta\left(\frac{d}{2}, \frac{d}{2}\right)$ distributed on $(-1,1)$. That is to say, they are random variables with mean 0 and variance $\frac{1}{d+1}$. Then the variables  $\sqrt{d+1} \cdot \rho_{ij}$ have mean zero and variance one. Ignoring for a moment that the $\rho_{ij}'s$ are not independent, taking $a_{ij}= \sqrt{d+1} \cdot \rho_{ij}$, and applying the result above, we obtain that $Y_d$ can be approximated by $O(-d) + \left(\sqrt{O\left(2\ln\left(\frac{d}{2}\right)\right)}\right)N(0,1)$, where $N(0,1)$ is a standard normal variable.\\

The above result is somewhat in line with the previous calculations for the mean and variance of $Y_d$. Moreover, simulating $\frac{Y_d + d}{\sqrt{2\ln\left(\frac{d}{2}\right)}}$, for dimensions 400, and 500, we obtain a distribution that is undistinguishable from a normal distribution (when performing a two-sample Kolmogorov-Smirnov goodness-of-fit hypothesis test) with mean $1.27$, and standard deviation $1.00$. As shown above the mean of $Y_d$ is not $O(-d)$, but $O\left(-d+\frac{1}{2}\ln\left(\frac{d}{2}\right)+1+\frac{\gamma}{2}\right)$. The following section includes the formal proof of the asymptotic normality of $Y_d$.

\begin{figure}
\centering
\subfloat[Normal variable, d=400]{\includegraphics[width=0.45\textwidth]{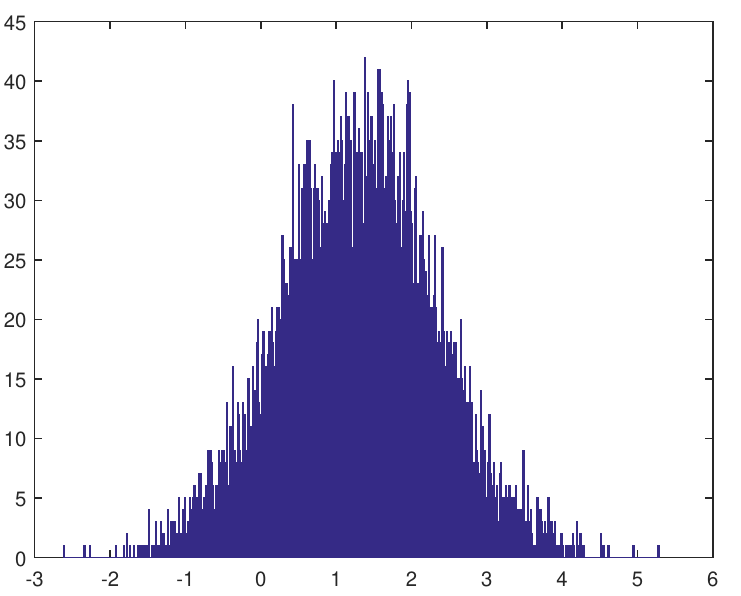}}
\subfloat[$\frac{Y_d + d}{\sqrt{2\ln\left(\frac{d}{2}\right)}}$, d=400]{\includegraphics[width=0.45\textwidth]{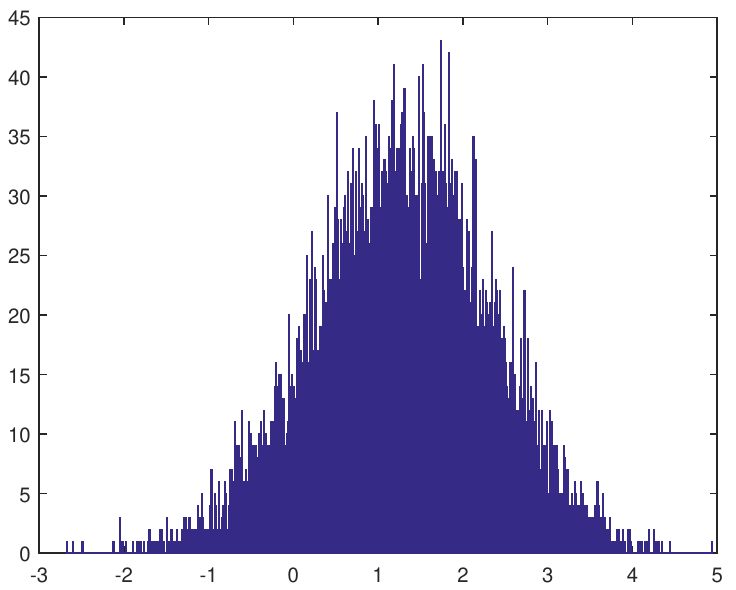}}\\
\subfloat[Normal variable, d=500]{\includegraphics[width=0.45\textwidth]{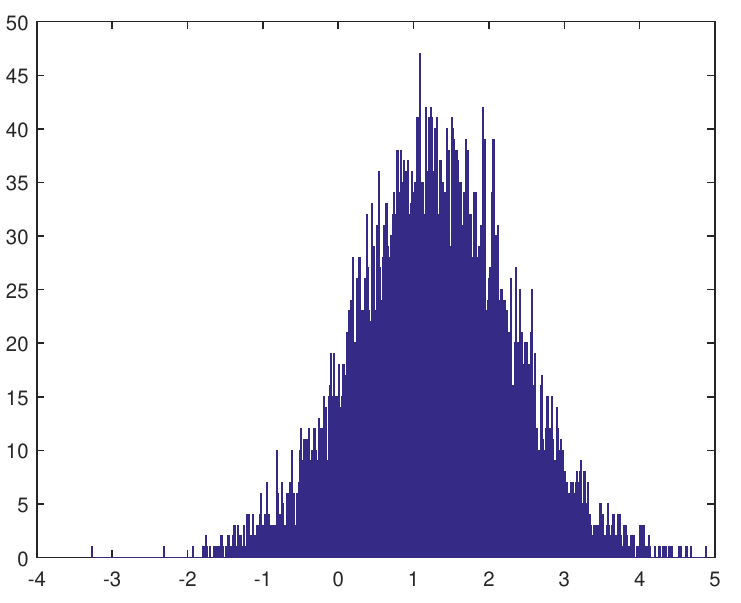}}
\subfloat[$\frac{Y_d + d}{\sqrt{2\ln\left(\frac{d}{2}\right)}}$, d=500]{\includegraphics[width=0.45\textwidth]{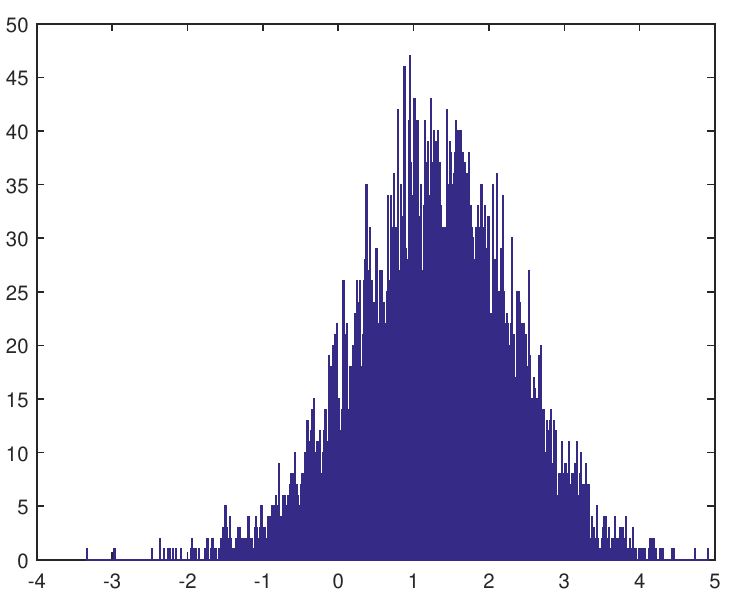}}
\caption{Simulations of normally distributed random variables (a,c) and scaled logarithm of the determinant of a random correlation matrix (b,d), for d=400 and d=500.}
\end{figure}

\subsection{The asymptotic normality of $Y_d$}

In this section we will prove the Lyapunov CLT of the logarithm of the determinant of a random correlation matrix:

\begin{equation}
\label{log_det}
Y_d=\ln D_d=\sum_{j=1}^{d-1} \ln B_j,
\end{equation}
where $B_j \sim Beta\left( \frac{j+1}{2}, \frac{d-j}{2}\right)$ are independent random variables. Note that $Y_d$ involves a sequence of sequences, thus a CLT theorem for triangular arrays is provided below.

\begin{thm}[Asymptotic normality]\label{asympt_normality}
For $d>0$, consider the uniform distribution over the space of $d \times d$ correlation matrices. Let $R_d$ be a random correlation matrix and let $D_d$ be its determinant. Moreover, let $Y_d$ be the logarithm of $D_d$. Then, as $d\rightarrow\infty$,
\[
\frac{Y_d +d+\frac{1}{2}\ln\left(\frac{d}{2}\right)+\frac{\gamma}{2}+1}{\sqrt{2\ln\left(\frac{d}{2}\right)}} \law N(0,1).
\]
\end{thm}

\begin{proof}
Let

\[
\mu_j=E(\ln B_j)=\psi\left( \frac{j+1}{2}\right) - \psi\left( \frac{d+1}{2}\right),
\]
and
\[
\sigma^2_j=Var(\ln B_j)=\psi_1\left( \frac{j+1}{2}\right) - \psi_1\left( \frac{d+1}{2}\right).
\]
where $\psi$ is the digamma function and $\psi_1$ is the trigamma function. The following inequality from \cite{Mitrinovic1970},

\begin{equation}
\label{ineq_psi1}
\frac{1}{x}<\psi_1(x)<\frac{1}{x-1},
\end{equation}
for any $x>1$, implies that

\[
\sigma^2_j<\frac{2}{j-1}-\frac{2}{d+1}<\infty.
\]
For $j=1$, $\psi_1(1)=\frac{\pi^2}{6}$ and
\[
\sigma^2_1<\frac{\pi^2}{6}-\frac{2}{d+1}<\infty.
\]
Thus $\sigma^2_j<\infty$, for all $j=1,\ldots,d-1$. Let

\[
s_{d-1}^2=\sum_{j=1}^{d-1}\sigma^2_j.
\]
It is immediate that the sequence $\frac{Y_d-E(Y_d)}{s_{d-1}}$ has mean 0 and variance 1. To prove its asymptotic normality, we will show that the Lyapunov condition is satisfied, that is, there exists $\delta>0$ such that

\begin{equation}
\label{lyapunov}
\frac{1}{s_{d-1}^{2+\delta}}\sum_{j=1}^{d-1}E\left( \left| \ln B_j -\mu_j \right|^{2+\delta} \right)\to 0,
\end{equation}
for $d\to\infty$.

We will prove \eqref{lyapunov} for $\delta=2$, hence investigating the fourth central moment of the random variable $\ln B_j$, for $j=1,\ldots,d-1$,

\[
E\left[ (\ln B_j-\mu_j)^4 \right]=k_4+3E\left[ (\ln B_j-\mu_j)^2 \right],
\]
where $k_4$ is the fourth cumulant of the distribution of $\ln B_j$, with $B_j \sim Beta\left( \frac{j+1}{2}, \frac{d-j}{2}\right)$. By letting $\psi_{n-1}$ denote the polygamma functions, then

\[
k_n=\psi_{n-1}\left(\frac{j+1}{2}\right)-\psi_{n-1}\left( \frac{d+1}{2}\right),
\]
and we can write
\[
E\left[ (\ln B_j-\mu_j)^4 \right]=\psi_3\left(\frac{j+1}{2}\right)-\psi_3\left( \frac{d+1}{2}\right)+3\psi_1\left(\frac{j+1}{2}\right)-3\psi_1\left( \frac{d+1}{2}\right).
\]
The authors of \cite{Guo_Qi2015} have shown that, for any $x>0$,
\[
\frac{(n-1)!}{(x+1)^n}+\frac{n!}{x^{n+1}}<|\psi_n(x)|<\frac{(n-1)!}{(x+\frac{1}{2})^n}+\frac{n!}{x^{n+1}}.
\]
By using this inequality, it follows that
\[
\begin{split}
\sum_{j=1}^{d-1} E\left( (\ln B_j-\mu_j)^4 \right)<& \sum_{j=1}^{d-1} \Bigg[\frac{16}{(j+2)^3}+\frac{96}{(j+1)^4}+\frac{16}{(d+2)^3}+\frac{96}{(d+1)^4}\\
&\quad \quad +\frac{6}{j+2}+\frac{12}{(j+1)^2}+\frac{6}{d+2}+\frac{12}{(d+1)^2}\Bigg]\\
&\equiv A_d.
\end{split}
\]
To bound the variance term in \eqref{lyapunov}, we employ another inequality provided in \cite{Guo_Qi2015}
\[
\frac{1}{x-\frac{1}{2}}-\frac{1}{12(x-\frac{1}{2})^2}<\psi_1(x)<\frac{1}{x-\frac{1}{2}},
\]
for $x>\frac{1}{2}$. we have that
\[
\psi_1\left( \frac{j+1}{2}\right) - \psi_1\left( \frac{d+1}{2}\right)>\frac{2}{j}-\frac{1}{3j^2}-\frac{2}{d},
\]
which gives that
\[
s_{d-1}^4>\left[\sum_{j=1}^{d-1}\left(\frac{2}{j}-\frac{1}{3j^2}-\frac{2}{d}\right)\right]^2\equiv B_d.
\]
We have that
\[
\frac{1}{s_{d-1}^{4}}\sum_{j=1}^{d-1}E\left( \left| \ln B_j -\mu_j \right|^{4} \right)<\frac{A_d}{B_d}<\frac{A'_d}{B_d},
\]
where
\[
A'_d= \sum_{j=1}^{d-1} \Bigg[\frac{32}{(j+2)^3}+\frac{192}{(j+1)^4}+\frac{12}{j+2}+\frac{24}{(j+1)^2}\Bigg]
\]
Proving that $\frac{A'_b}{B_d}\to0$ for $d\to\infty$ completes then the proof of the theorem. Note that, when letting $d\to\infty$, the denominator and numerator contain convergent p-series. The dominant terms are the harmonic series, and by \eqref{harmonic_no}, they can be approximated by $\ln(d-1)+\gamma$. The squared sum in the denominator gets the limit to zero.
\end{proof}

\subsection{$Y_d$'s moment generating function}

The moment generating function of $Y_d$ can be calculated as follows:

\begin{align}\label{MYd}
M_{Y_d}(t)&=E\left( e^{t\cdot \sum_{j=1}^{d-1} \ln B_j}  \right) = \prod_{j=1}^{d-1} \frac{B\left(t+\frac{j+1}{2},\frac{d-j}{2} \right)}{B\left(\frac{j+1}{2},\frac{d-j}{2} \right)},
\end{align}

where $B(\alpha,\beta)$ is a Beta function.

Using that $\frac{\Gamma(z+a)}{\Gamma(z+b)} \approx z^{a-b}$, as $z\longrightarrow+\infty$ and $S = \frac{d+1}{2}$, and Proposition \ref{prop2} we obtain\footnote{The product is dominated by the higher order terms.}:

\begin{align*}
M_{Y_d}(t)&= \prod_{j=1}^{d-1} \frac{\Gamma\left(\frac{j+1}{2}+t \right)}{\Gamma \left(\frac{j+1}{2}\right)} \cdot \left[\frac{\Gamma(S)}{\Gamma(t+S)} \right]^{d-1}\\
& \approx \prod_{j=1}^{d-1} \left(\frac{j+1}{2}\right)^t
 \cdot \frac{2^{t(d-1)}}{(d+1)^{t(d-1)}}\\
&= \left[ \frac{d!}{(d+1)^{d-1}}\right]^t = \left[E(D_d)\right]^t.
\end{align*}

\noindent In this case, the first derivative of $M_{Y_d}(t)$ evaluated at $t=0$ equals $\ln{E(D_d)}$, hence $\ln{E(D_d)} = E(\ln{D_d})$.
This suggests again a linear behaviour caused by the degenerate distribution of the scaled determinant.

\section{Discussion}

This research is concerned with the asymptotic distribution of the determinant~$D_d$, of a
random correlation matrix sampled from a uniform distribution over the space of $d \times d$ correlation
matrices. We proved that the $d^{th}$ root of the determinant has a degenerate limiting distribution. Several spin-off results proven along the way include expressions for the first two moments of $D_d$ and results about the asymptotic behaviour of the of the $d^{th}$ root of the first two moments of $D_d$. An interesting connection between $E(D_d)$ and elementary norms is established. Another link with the law of the determinant of general random matrices is also investigated.\\
\indent When looking at the logarithm of the determinant of the correlation matrix, we are actually investigating a sum of independent random variables. Using this, the asymptotic normality of the logarithm of the determinant of a random correlation matrix was established.

\section{Appendix}

\noindent The proof of Corollary \ref{main_cor} follows the same lines as the proof of Theorem \ref{mainth}.

\begin{repcor}{main_cor}
Consider the $d \times d$ correlation matrices such that each $\rho_{ij}$ marginally follows a Beta distribution on $(-1, 1)$, with both parameters equal to $\left(\eta - 1 + \frac{d}{2}\right)$, for $\eta > 0$. Let $R_d$ be such a matrix and $D_{d}$ its determinant. Then $D_d^{1/d} \to \frac{1}{e}$, as $d\rightarrow\infty$, in probability.
\end{repcor}

\begin{proof}
Consider the expression for $D_d$ given in Corollary 1. Let $B_j$ denote the $Beta$ variables from the product; $B_j$'s are independent, so the expectation of  $D_d^{1/d}$ can be calculated as a product of expectations of $Beta$ variables raised to the power $\frac{1}{d}$.\\
 \indent Let $\alpha = \eta + \frac{j-1}{2}$, $\beta = \frac{d-j}{2}$, $S = \alpha+\beta=\eta+\frac{d-1}{2}$. Let $B (\alpha,\beta)$ denote the Beta function, and $\Gamma(\cdot)$ denote the Gamma function. Then:

\begin{align*}
E\left(D_d^{1/d}\right) &= E \left(\prod_{j=1}^{d-1} B_j^{1/d}\right)\\
&=\prod_{j=1}^{d-1} E\left(B_j^{1/d}\right)\\
&= \prod_{j=1}^{d-1}\frac{B(\alpha+\frac{1}{d}, \beta)}{B(\alpha,\beta)} = \prod_{j=1}^{d-1} \frac{\Gamma(\alpha + \frac{1}{d}) \cdot \Gamma(\alpha + \beta)}{\Gamma(\alpha ) \cdot \Gamma(\alpha + \beta + \frac{1}{d})}\\
&= \left[ \frac{\Gamma(S)}{\Gamma(S+\frac{1}{d})} \right]^{d-1} \cdot  \prod_{j=1}^{d-1} \frac{\Gamma(\eta +\frac{j-1}{2}+\frac{1}{d})}{\Gamma(\eta+\frac{j-1}{2})}.
\end{align*}

\noindent By \eqref{frac_gamma},

\begin{align*}
E\left(D_d^{1/d}\right) &= \left(S^{-\frac{1}{d}}\right)^{d-1} \cdot \prod_{j=1}^{d-1} \left( \frac{2\eta+j-1}{2}\right)^{\frac{1}{d}}\\
&=\frac{2^{\frac{d-1}{d}}}{(2\eta +d-1)^{\frac{d-1}{d}}} \cdot \frac{[(2\eta +d-2)!]^{\frac{1}{d}}}{[(2\eta -1 )!]^{\frac{1}{d}}}\cdot \frac{1}{2^{\frac{d-1}{d}}}.
\end{align*}

\noindent Using \eqref{stirling} and rearranging the terms gives

\begin{align*}
E \left(D_d^{1/d}\right) &= \frac{1}{[(2\eta -1 )!]^{\frac{1}{d}}}\cdot \frac{1}{e^{\frac{2\eta+d-2}{d}}} \cdot \left(\sqrt{2\pi}\right)^{\frac{1}{d}}\cdot \left(2\eta +d-2 \right)^{\frac{1}{2d}}\\[10pt]
&\cdot \frac{2\eta +d-2}{2\eta+d-1}\cdot \left(2\eta+d-2\right)^{\frac{2\eta-2}{d}}\cdot\left(2\eta+d-1 \right)^{\frac{1}{d}}.
\end{align*}
Then $\displaystyle{\lim_{d\longrightarrow\infty}} E\left(D_d^{1/d}\right)= \frac{1}{e}$. \\
In the same fashion we show that $\displaystyle{\lim_{d\longrightarrow\infty}} E(D_d^{2/d})= \frac{1}{e^2}$. To summarize we have:
\begin{align*}
\lim_{d\longrightarrow\infty} E\left(D_d^{1/d}\right)&= \frac{1}{e}\\
\lim_{d\longrightarrow\infty} \left[E\left(D_d^{1/d}\right)\right]^2 &= \frac{1}{e^2}\\
\lim_{d\longrightarrow\infty} E\left[\left(D_d^{1/d}\right)^2\right] &= \frac{1}{e^2}.
\end{align*}

\noindent The above equalities prove $\displaystyle{\lim_{d\longrightarrow\infty}} var\left(D_d^{1/d}\right) = 0$ and the Corollary.
\end{proof}

\noindent Let us now prove Propositions \ref{prop1}.

\begin{repprop}{prop1}
For a uniform distribution over the space of $d \times d$ correlation matrices, the marginal
distribution of each correlation follows $Beta(d/2, d/2)$ distribution on $(-1, 1)$ and the determinant $D_d$ of
a correlation matrix follows the expression from equation (\ref{det_as1prod}). Then:
\begin{enumerate}
\item $\displaystyle{\lim_{d\longrightarrow\infty}} [E(D_d)]^{\frac{1}{d}}= \frac{1}{e} = \lim_{d\longrightarrow\infty} E(D_d^{1/d})$
\item $\displaystyle{\lim_{d\longrightarrow\infty}} [var(D_d)]^{\frac{1}{d}} = \frac{1}{e^2} = \lim_{d\longrightarrow\infty}\left[E(D_d)^{\frac{1}{d}}\right]^2$.
\end{enumerate}
\end{repprop}

\begin{proof}

The proof of the second equality in both statements can be found in the proof of Theorem \ref{mainth}.
To prove the first equality in 1. calculate $E(D_d)$ as follows:

\begin{equation}
\label{exp_det}
E(D_d)=\prod_{j=1}^{d-1}\frac{\frac{j+1}{2}}{\frac{d+1}{2}} = \frac{d!}{(d+1)^{d-1}}.
\end{equation}

\noindent By \eqref{stirling}, we obtain

\begin{align*}
\left[E(D_d)\right]^{\frac{1}{d}} &= \left[ \frac{d!}{(d+1)^{d-1}}\right]^{\frac{1}{d}}\\
&= \frac{(2 \pi)^{\frac{1}{d}}\cdot d^{\frac{1}{2} \cdot \frac{1}{d}}\cdot d^{d\frac{1}{d}}}{e^{d \cdot \frac{1}{d}} \cdot (d+1)^{\frac{d-1}{d}}}\cdot \left[ 1+O\left(\frac{1}{d}\right) \right]^{\frac{1}{d}}\\
& = (2 \pi)^{\frac{1}{d}} \cdot \frac{d}{d+1} \cdot d^{\frac{1}{2d}}\cdot (d+1)^{\frac{1}{d}} \cdot \frac{1}{e}\cdot \left[ 1+O\left(\frac{1}{d}\right) \right]^{\frac{1}{d}}.
\end{align*}

\noindent Then $\displaystyle\lim_{d\longrightarrow\infty} [E(D_d)]^{\frac{1}{d}} = \frac{1}{e}$.

\noindent Using \eqref{stirling} again, we can calculate:

\begin{align*}
\left[E(D_d)\right]^{\frac{2}{d}} &=\frac{(2 \pi)^{\frac{2}{d}}\cdot d^{\frac{1}{2} \cdot \frac{2}{d}}\cdot d^{d\frac{2}{d}}}{e^{d \cdot \frac{2}{d}} \cdot (d+1)^{\frac{2(d-1)}{d}}}\cdot \left[ 1+O\left(\frac{1}{d}\right) \right]^\frac{2}{d}\\
& = (2 \pi)^{\frac{1}{d}} \cdot \left(\frac{d}{d+1}\right)^{2} \cdot d^{\frac{1}{d}}\cdot (d+1)^{\frac{2}{d}} \cdot \frac{1}{e^2}\cdot \left[ 1+O\left(\frac{1}{d}\right) \right]^\frac{2}{d}
\end{align*}

\noindent which proves  $\displaystyle{\lim_{d\longrightarrow\infty}} \left[E(D_d)^{\frac{1}{d}}\right]^2= \frac{1}{e^2}$.

\noindent To prove Relation 2. from Proposition \ref{prop1} we start from the expression of the determinant $D_d$ given in equation (\ref{det_as1prod}) and consider $(d-1)$ independent, Beta distributed random variables, $B_j \sim Beta\left(\frac{j+1}{2}, \frac{d-j}{2}\right), j=1, \cdots, d-1$. Then

\begin{align}
\label{aici}
var(D_d) &= var\left(\prod_{j=1}^{d-1} B_j\right) = \prod_{j=1}^{d-1} \left( \frac{\frac{j+1}{2} \cdot \frac{j+3}{2}}{\frac{d+1}{2} \cdot \frac{d+3}{2}}\right) - \prod_{j=1}^{d-1} \left(\frac{\frac{(j+1)^2}{4}}{\frac{(d+1)^2}{4}}\right)\nonumber\\
&= \frac{1}{(d+1)^{d-1} \cdot (d+3)^{(d-1)}} \cdot \prod_{j=1}^{d-1}(j+1) \cdot \prod_{j=1}^{d-1}(j+3) - \frac{1}{(d+1)^{2(d-1)}} \cdot \prod_{j=1}^{d-1}(j+1)^2\nonumber\\
&= \frac{d!(d+2)!}{6(d+1)^{d-1}(d+3)^{d-1}} - \frac{(d!)^2}{(d+1)^{2(d-1)}}\\
&= \frac{d!}{(d+1)^{(d-1)}}\left[ \frac{(d+2)!}{6(d+3)^{d-1}} - \frac{d!}{(d+1)^{d-1}}\right]\nonumber\\
&= \frac{(d+1)!}{(d+1)^{d}}\left[ \frac{(d+3)!}{6(d+3)^{d}} - \frac{(d+1)!}{(d+1)^{d}}\right]\nonumber.
\end{align}
\noindent By \eqref{stirling},

\begin{align*}
var(D_d) &= \frac{\sqrt{2\pi (d+1)}(d+1)^{d+1}}{e^{d+1}(d+1)^d}\left[ \frac{\sqrt{2\pi (d+3)}(d+3)^{d+3}}{e^{d+3}6(d+3)^d} - \frac{\sqrt{2\pi (d+1)}(d+1)^{d+1}}{e^{d+1}(d+1)^d}\right]\cdot \left[ 1+O\left(\frac{1}{d}\right) \right]^2\\
&= \frac{\sqrt{2\pi (d+1)}(d+1)}{e^{d+1}}\left[ \frac{\sqrt{2\pi (d+3)}(d+3)^{3}}{6 \cdot e^{d+3}} - \frac{\sqrt{2\pi (d+1)}(d+1)}{e^{d+1}}\right]\cdot\left[ 1+O\left(\frac{1}{d}\right) \right]^2\\
&= \frac{2\pi \cdot (d+1)^{3/2} \cdot (d+1)^{1/2}}{e^{2(d+1)}} \left[ \frac{\sqrt{\frac{d+3}{d+1}}}{6} \cdot \frac{(d+3)^3}{e^2} -(d+1)\right]\cdot\left[ 1+O\left(\frac{1}{d}\right) \right]^2\\
&= \frac{\pi \cdot (d+1)^{3/2}}{3e^{2(d+2)}}\left[ (d+3)^{7/2} -6e^2(d+1)^{3/2} \right]\cdot \left[ 1+O\left(\frac{1}{d}\right) \right]^2.
\end{align*}

\noindent By raising to the $1/d$ power and taking the limit when $d \to \infty$ we obtain $\displaystyle{\lim_{d\longrightarrow\infty}} [var(D)]^{\frac{1}{d}} = \frac{1}{e^2}$.
\end{proof}

\noindent
The proof of Corollary \ref{mom_det} is provided below.

\begin{repcor}{mom_det}
The first two moments of the distribution of the determinant of a $d \times d$ random correlation matrix can be written as follows:

\begin{enumerate}
\item $E(D_d) = \left[E\left(D_d^{1/d}\right)\right]^d + O\left(\frac{1}{d}\right)$
\item $var(D_d) = \left[E(D_d)\right]^2 + O\left(\frac{1}{d}\right)$.
\end{enumerate}
\end{repcor}

\begin{proof}
\noindent By \eqref{exp_det_1/d}

\[
E \left(D_d^{1/d}\right)^d = \sqrt{2\pi d} \cdot \frac{d^d}{(d+1)^{d-1}} \cdot e^{-d}\cdot\left[ 1+O\left(\frac{1}{d}\right) \right]^{d-1}.
\]
By \eqref{exp_det} and \eqref{stirling},

\[
E\left(D_d \right)= \sqrt{2\pi d} \cdot \frac{d^d}{(d+1)^{d-1}} \cdot e^{-d}\cdot\left[ 1+O\left(\frac{1}{d}\right) \right].
\]
It will follow that

\[
E\left(D_d \right) = E \left(D_d^{1/d}\right)^d +O\left(\frac{1}{d}\right).
\]
For the second relationship, notice that, by \eqref{exp_det} and \eqref{stirling},

\[
\begin{split}
\left[E(D_d)\right]^2 = \prod_{j=1}^{d-1}\frac{\frac{j+1}{2}}{\frac{d+1}{2}} &= \frac{d!\cdot d!}{(d+1)^{2(d-1)}}\\
&=\frac{ 2\pi d\cdot d^{2d}\cdot e^{-2d}\cdot}{(d+1)^{2(d-1)}}\cdot \left[ 1+O\left(\frac{1}{d}\right) \right]^2.
\end{split}
\]
Furthermore,
\[
\begin{split}
var(D_d) &= var\left(\prod_{j=1}^{d-1} B_j\right) = \prod_{j=1}^{d-1}  \frac{ (j+1) \cdot (d-j)}{2(d+1)^2\cdot(d+3)}\\
&= \frac{1}{2^{d-1}\cdot(d+1)^{2(d-1)}\cdot(d+3)^{d-1}}\prod_{j=1}^{d-1}(j+1)\cdot(d-j)\\
&= \frac{(d-1)!\cdot d!}{2^{d-1}\cdot(d+1)^{2(d-1)}\cdot(d+3)^{d-1}}\\
&= \frac{2\pi \sqrt{d\cdot(d-1)} \cdot (d-1)^{d-1}\cdot e^{-(d-1)}\cdot d^d\cdot e^{-d}}{2^{d-1}\cdot(d+1)^{2(d-1)}\cdot(d+3)^{d-1}}\cdot\left[ 1+O\left(\frac{1}{d}\right) \right]^2\\
&= \left[E(D_d)\right]^2 \cdot\left[ 1+O\left(\frac{1}{d}\right) \right]^2
\end{split}
\]
This proves the corollary.
\end{proof}

\noindent The relations from Proposition \ref{prop2} find their derivations in the proof of Proposition \ref{prop1} above.

\section*{Acknowledgements}

The author gratefully acknowledges fruitful discussions with, and helpful comments from Frank Redig and Roger Cooke.

\bibliographystyle{plain}

\begin{thebibliography}{99}

\bibitem{Apostol_1976}
Apostol, T.M. Introduction to Analytic Number Theory. {\em New York: Springer Verlag}, 1976.

\bibitem{Bedford_2002}
Bedford, T.J., Cooke, R.M.,  Vines - A New Graphical Model for Dependent Random
  Variables, {\em Annals of Statistics}, 30(4):1031--1068, 2002.

\bibitem{Cooke_97}
Cooke, R.M., Markov and {E}ntropy {P}roperties of {T}ree and {V}ine-{D}ependent
  {V}ariables, In {\em Proceedings of the Section on Bayesian Statistical Science,
  American Statistical Association}, 1997.

\bibitem{Guo_Qi2015}
Guo, B.-N., Qi, F., Zhao, J.-L., Luo, Q.-M., Sharp inequalities for polygamma functions {\em Mathematica Slovaca}, 65(1):103-120, 2015.

\bibitem{Hardy_L_P}
Hardy, G.H., Littlewood, J.E., Polya,G. Inequalities, {\em Cambridge University Press, Cambridge Mathematical Library}, 1988.

\bibitem{Holmes91}
Holmes, R.B. On random corelation matrices, {\em {SIAM} {M}atrix {A}nal. {A}ppl,}, 1991.

\bibitem{Joe2006}
Joe, H. Generating random correlation matrices based on partial correlations, {\em Journal of Multivariate Analysis}, 97:2177--2189, 2006.

\bibitem{Johnson1980}
Johnson, D.G., Welcht, W.J. The generation of pseudo-random correlation matrices, {\em J. Statist.Comput.Simul.}, 22:55--69, 1980.

\bibitem{Kurowicka2006LA}
Kurowicka, D., Cooke,  R.M. Completion problem with partial correlation vines, {\em Linear Algebra and its Applications}, 2006.

\bibitem{Laforgia2012}
Natalini, P., Laforgia, A.  On the asymptotic expansion of a ratio of gamma functions. {\em Journal of mathematical analysis and applications}, 389:833 -- 837, 2012.

\bibitem{Lewandowski2009}
Lewandowski, D., Kurowicka, D., Joe, H. Generating random correlation matrices based on vines and exptended
  onion method, {\em Journal of Multivariate Analysis}, 100:1989 -- 2001, 2009.

\bibitem{Mitrinovic1970}
Mitrinovic, D.S. Analytic inequalities, {\em New York: Springer Verlag}, 1970.

\bibitem{NadarajahGupta}
Nadarajah, S., Gupta, A.K. Characterizations of the beta distribution, {\em {COMMUNICATIONS IN STATISTICS}: Theory and Methods}, 33(12):2941 -- 2957, 2004.

\bibitem{erdelyi1951}
Erd\'{e}lyi, A., Tricomi, F. G.  The asymptotic expansion of a ratio of gamma functions, {\em Pacific Journal of Mathematics}, 1(1):133 -- 142, 1951.

\bibitem{Nguyen2012}
Vu, H.V. Nguyen, H.H. Random matrices: law of the determinant, {\em arXiv.org/math/arXiv:1112.0752}, 2012.

\bibitem{Qiu2004}
Qiu, W. Separation index, variable selection and sequential algorithm for
  cluster analysis,  {\em {P}h.{D}. {T}hesis, {D}epartment of {S}tatistics, {U}niversity of {B}ritish {C}olumbia}, 2004.

\end{thebibliography}

\end{document}